\documentclass[12pt,twoside]{amsart}
\usepackage{amssymb,amsmath,amsthm, amscd, enumerate, mathrsfs}
\usepackage[all]{xy}

\title{Direct images of relative pluricanonical bundles}
\author{Osamu Fujino} 
\date{2015/4/22, version 0.21}
\address{Department of Mathematics, Graduate School of Science, 
Kyoto University, Kyoto 606-8502, Japan}
\email{fujino@math.kyoto-u.ac.jp}

\newcommand{\Exc}[0]{{\operatorname{Exc}}}

\newtheorem{thm}{Theorem}[section]
\newtheorem{lem}[thm]{Lemma}
\newtheorem{cor}[thm]{Corollary}

\newtheorem{conj}[thm]{Conjecture}

\theoremstyle{definition}
\newtheorem{defn}[thm]{Definition}
\newtheorem{rem}[thm]{Remark}
\newtheorem*{ack}{Acknowledgments} 
\newtheorem{say}[thm]{}
\newtheorem{step}{Step}
\begin{document}

\begin{abstract}
We discuss the local freeness and the 
numerical semipositivity of 
direct images of relative pluricanonical 
bundles for surjective morphisms between smooth projective varieties 
with connected fibers. 
We give a desirable semipositivity theorem under the 
assumption that 
the geometric generic fiber has a good minimal model. 
\end{abstract}

\maketitle

\tableofcontents 

\section{Introduction}\label{sec1}

By Griffiths's theory of variations of Hodge structure (see 
\cite{griffiths}), we have: 

\begin{thm}[Griffiths]\label{thm1.1} 
Let $f:X\to Y$ be a smooth morphism between smooth projective 
varieties. Then $f_*\omega_{X/Y}$ is a nef locally free 
sheaf. \end{thm}

Before we go further, 
let us recall the definition of nef (numerically semipositive) 
locally free sheaves. 
 
\begin{defn}[Nef locally free sheaves]\label{def1.2} 
Let $\mathcal E$ be a locally free sheaf of finite rank on a 
complete algebraic variety $V$. 
Then $\mathcal E$ is called {\em{nef}} if $\mathcal E=0$ or  
$\mathcal O_{\mathbb P_V(\mathcal E)}(1)$ is 
nef on $\mathbb P_V(\mathcal E)$. 
A nef locally free sheaf $\mathcal E$ 
was originally 
called a {\em{$($numerically$)$ semipositive}} locally free sheaf in the literature. 
\end{defn}

Precisely speaking, Griffiths proved that $f_*\omega_{X/Y}$ is 
semipositive in the sense of Griffiths and his result is sharper than 
Theorem \ref{thm1.1}. 
Moreover, Berndtsson proved that $f_*\omega_{X/Y}$ is 
semipositive in the sense of Nakano by $L^2$ method 
(see \cite[Theorem 1.2]{berndtsson}). 
Unfortunately, 
Theorem \ref{thm1.1} is not so useful for various geometric applications 
since we need the smoothness of $f$. 
In \cite{kawamata1}, Kawamata proved 
Theorem \ref{thm1.3}, which is a natural 
generalization of Theorem \ref{thm1.1}, by using 
the theory of variations of Hodge structure (see \cite[Theorem 5]{kawamata1}). 

\begin{thm}[Fujita, Zucker, Kawamata, $\cdots$]
\label{thm1.3} Let $f:X\to Y$ be a surjective morphism between 
smooth projective varieties with connected fibers. 
Then there exists a generically finite morphism $\tau:Y'\to Y$ from 
a smooth projective variety $Y'$ with 
the following property. 
Let $X'$ be any resolution of the main component of $X\times _Y Y'$. 
Then $f'_*\omega_{X'/Y'}$ is a nef locally free sheaf, 
where $f': X'\to X\times _Y Y'\to Y'$. 
\end{thm}
For the details of Kawamata's original approach and 
various generalizations, see \cite[Theorems 3.1, 3.4, and 
3.9]{fujino-higher}, 
\cite[Theorem 1.1 and Theorem 1.3]{fujino-fujisawa}, 
and \cite[Corollary 2, Theorem 2, and 
Theorem 3]{fujino-fujisawa-saito}.
Theorem \ref{thm1.3} has already played a crucial role 
in the study of higher-dimensional algebraic varieties. 
For some geometric applications, we have to 
treat $f_*\omega^{\otimes m}_{X/Y}$ or 
$f'_*\omega^{\otimes m}_{X'/Y'}$, where $m$ is a positive integer. 
It is well known that Viehweg proved that 
$f_*\omega^{\otimes m}_{X/Y}$ is always 
weakly positive 
for every positive integer $m$ 
in Theorem \ref{thm1.3} (see \cite[Theorem III]{viehweg}). 
His original 
proof of the weak positivity of $f_*\omega^{\otimes m}_{X/Y}$ 
uses his mysterious covering trick and Theorem \ref{thm1.3} 
(see \cite[\S 5]{viehweg}). 

By the way, Theorem \ref{thm1.1} can be generalized as follows. 
\begin{thm}\label{thm1.4}
Let $f:X\to Y$ be a smooth morphism between smooth 
projective varieties. 
Then $f_*\omega^{\otimes m}_{X/Y}$ is a nef 
locally free sheaf for every positive integer $m$. 
\end{thm}

We give a proof of Theorem \ref{thm1.4} based on Siu's invariance of 
plurigenera (see \cite[Corollary 0.2]{siu} and \cite[Theorem 1]{paun}) 
and the effective freeness in \cite{popa-schnell} 
(see \cite[Theorem 1.4]{popa-schnell}). 
Note that Siu's invariance of plurigenera is not Hodge theoretic. 
It is a very 
clever application of the Ohsawa--Takegoshi $L^2$ extension theorem. 
We have no Hodge theoretic characterization of $f_*\omega^{\otimes m}_{X/Y}$ 
in Theorem \ref{thm1.4} when $m\geq 2$. 
By Theorem \ref{thm1.3} and Theorem \ref{thm1.4}, 
it is natural to consider: 

\begin{conj}[Semipositivity of direct images of relative pluricanonical 
bundles]\label{conj1.5}
Let $f:X\to Y$ be a surjective morphism between 
smooth projective varieties with connected fibers. 
Then there exists a generically finite morphism 
$\tau:Y'\to Y$ from a smooth projective variety $Y'$ 
with the following property. 
Let $X'$ be any resolution of the main component of $X\times _Y Y'$ 
sitting in the following commutative diagram: 
$$
\xymatrix{X' \ar[r]\ar[d]_{f'}& X\ar[d]^f\\
Y'\ar[r]_{\tau} &Y. 
}
$$ 
Then $f'_*\omega^{\otimes m}_{X'/Y'}$ is a nef 
locally free sheaf for every positive integer $m$. 
\end{conj}

Conjecture \ref{conj1.5} can be seen as a correct formulation of 
Fujita's very naive conjecture:~\cite[Conjecture Wa$_m$]{fujita}. 
Note that \cite{fujita} contains 17 conjectures and 
that \lq\lq Wa\rq\rq\ means 13th in \cite{fujita}.  

The main purpose of this paper is to prove: 

\begin{thm}[Main theorem]\label{thm1.6}
Let $f:X\to Y$ be a surjective morphism between 
smooth projective varieties with connected fibers. 
Assume that the geometric generic fiber $X_{\overline \eta}$ of $f:X\to Y$ 
has a good minimal model. 
Then there exists a generically finite morphism 
$\tau:Y'\to Y$ from a smooth projective variety $Y'$ 
with the following property. 
Let $X'$ be any resolution of the main component of $X\times _Y Y'$ 
sitting in the following commutative diagram:  
$$
\xymatrix{X' \ar[r]\ar[d]_{f'}& X\ar[d]^f\\
Y'\ar[r]_{\tau} &Y. 
}
$$ 
Then $f'_*\omega^{\otimes m}_{X'/Y'}$ is a nef 
locally free sheaf for every positive integer $m$. 
\end{thm}

We note that $X_{\overline \eta}$ has a good minimal model 
if $\dim X_{\overline \eta}-\kappa (X_{\overline \eta})\leq 3$ 
(see \cite{bchm}, \cite{lai}, and Theorem \ref{thm3.8}). 
In particular, $X_{\overline \eta}$ has 
a good minimal model if $X_{\overline \eta}$ is of general type. 
Theorem \ref{thm1.6} reduces Conjecture \ref{conj1.5} to 
the good minimal model conjecture for geometric 
generic fibers. 
Of course, it is highly desirable to prove Conjecture \ref{conj1.5} 
without any extra assumptions. 
Our proof of 
Theorem \ref{thm1.6} is geometric and 
does not use the theory of variations of Hodge structure. 
We do not even use $L^2$ methods in the proof of Theorem \ref{thm1.6}. 
Our proof of Theorem \ref{thm1.6} in this paper is minimal model theoretic. 
Anyway, Theorem \ref{thm1.4} and Theorem \ref{thm1.6} strongly support 
Conjecture \ref{conj1.5}. 

\begin{rem}
If $Y$ is a curve in Conjecture \ref{conj1.5}, 
then Kawamata proved that 
$f_*\omega^{\otimes m}_{X/Y}$ is a nef locally free sheaf 
for every positive integer $m$  
(see \cite[Theorem 1]{kawamata2}). 
We also note that Viehweg's weak positivity 
of $f_*\omega^{\otimes m}_{X/Y}$ (see \cite[Theorem III]{viehweg}) 
implies that $f_*\omega^{\otimes m}_{X/Y}$ is 
nef when $Y$ is a curve. 
\end{rem}

We sketch the proof of Theorem \ref{thm1.6} for the reader's 
convenience. 

\begin{say}[Outline of the proof of Theorem \ref{thm1.6}]
We take a weak semistable reduction 
$f^\dag: X^\dag\to Y'$ in the sense of Abramovich--Karu. 
Then we take a good minimal model $\widetilde f:\widetilde X\to Y'$ of $f^\dag: 
X^\dag\to Y'$. 
Let $P$ be an arbitrary point of $Y'$ and let $C$ be a smooth 
curve on $Y'$ such that 
$P\in C$ and that $C=H_1\cap H_2\cap  \cdots \cap H_{\dim Y'-1}$, where 
$H_i$ is a general very ample 
Cartier divisor for every $i$. 
Then we can prove that 
$\widetilde X_C=\widetilde X\times _{Y'} C$ is a normal 
variety with only canonical singularities. 
Therefore, we obtain that $\widetilde f$ is flat and $\dim H^0(\widetilde X_y, 
\mathcal O_{\widetilde X}(mK_{\widetilde X/Y'})|_{\widetilde X_y})$ is independent of 
$y\in Y'$ for every positive integer $m$. 
This implies that $f'_*\omega^{\otimes m}_{X'/Y'}$ is locally free 
for every positive integer $m$. 
Once we establish 
the local freeness of $f'_*\omega^{\otimes m}_{X'/Y'}$, 
the nefness 
of $f'_*\omega^{\otimes m}_{X'/Y'}$ easily follows from 
the effective freeness by Popa--Schnell and 
Viehweg's fiber product trick. 
As explained above, a key point of 
the proof of Theorem \ref{thm1.6} is to construct 
a good minimal model 
$\widetilde f: \widetilde X\to Y'$ which behaves well 
under the base change by $C\hookrightarrow Y'$. 
Our proof of Theorem \ref{thm1.6} is not Hodge theoretic.  
\end{say}

After the author circulated a preliminary version of 
this paper, Mihai P\u aun and Shigeharu Takayama 
informed him of their new preprint \cite{paun-takayama}, 
where they prove various semipositivity theorems by $L^2$ methods. 
Their approach is completely different from ours. 
For the details, we recommend the reader to see \cite{paun-takayama} 
(see also Takayama's more recent results in \cite{takayama}). 

We summarize the contents of this paper. 
In Section \ref{sec2}, we collect some basic definitions and 
results for the reader's convenience. 
In Section \ref{sec3}, we discuss the relationship 
between relative good minimal models and good minimal models of 
fibers. 
In Section \ref{sec4}, we prove the local freeness of 
direct images of relative 
pluricanonical bundles in Theorem \ref{thm1.6} after taking 
a weak semistable reduction. In order to prove the local freeness, 
we take a relative good minimal model of the weak semistable 
reduction. Therefore, we need the assumption that 
the geometric generic fiber has a good minimal model. 
In Section \ref{sec5}, we prove the numerical 
semipositivity (nefness) in 
our main theorem:~Theorem \ref{thm1.6}. 
The proof is an easy application of the effective freeness obtained 
by Popa--Schnell (see \cite[Theorem 1.4]{popa-schnell}) and 
Viehweg's fiber product trick (see \cite[(3.4)]{viehweg}). 

\begin{ack}
The author was partially supported by Grant-in-Aid for 
Young Scientists (A) 24684002 from JSPS. 
He thanks Yoshinori Gongyo and Professors 
Shigefumi Mori and Shigeharu Takayama for useful comments. 
He also thanks Professors Mihai P\u aun and Shigeharu Takayama 
for sending him their new preprint \cite{paun-takayama}. 
Finally, he thanks Jinsong Xu for pointing out a mistake. 
\end{ack}

We will work over $\mathbb C$, the complex number field, throughout this paper. 
We will freely use the standard notations and results of 
the minimal model program as in 
\cite{kollar-mori}, \cite{fujino-fund} and \cite{fujino-foundation}. 
We recommend the reader to see \cite[\S 5]{mori} and \cite[Section 5]{fujino-remarks} 
for the details of Theorem \ref{thm1.3} and 
various related topics. 

\section{Preliminaries}\label{sec2}

In this section, we collect some basic notations and 
results for the reader's convenience. For the details, see \cite{kollar-mori}, 
\cite{fujino-fund}, and \cite{fujino-foundation}. 

\begin{say}[Dualizing sheaves and canonical divisors]\label{say2.1}
Let $X$ be a normal quasi-projective variety. 
Then we put $\omega_X=\mathcal H ^{-\dim X}(\omega^\bullet _X)$, 
where $\omega^\bullet _X$ is the dualizing complex of $X$, 
and call $\omega_X$ the {\em{dualizing sheaf}} of $X$. 
We put $\omega_X\simeq \mathcal O_X(K_X)$ and call $K_X$ the 
{\em{canonical divisor}} of $X$. 
Note that $K_X$ is a well-defined Weil divisor on $X$ up to the linear 
equivalence. 
Let $f:X\to Y$ be a morphism between Gorenstein varieties. 
Then we put $\omega_{X/Y}=\omega_X\otimes f^*\omega^{\otimes -1}_Y$ 
and call it the relative canonical bundle of $f:X\to Y$. 
\end{say}

\begin{say}[Singularities of pairs]\label{say2.2} 
Let $X$ be a normal variety and let $\Delta$ be an effective $\mathbb Q$-divisor 
on $X$ such that $K_X+\Delta$ is $\mathbb Q$-Cartier. 
Let $f:Y\to X$ be a resolution of singularities. 
We write $$K_Y=f^*(K_X+\Delta)+\sum _i a_i E_i$$ and 
$a(E_i, X, \Delta)=a_i$. 
Note that the {\em{discrepancy}} $a(E, X, \Delta)\in \mathbb Q$ can be 
defined for every prime divisor $E$ {\em{over}} $X$. 
If $a(E, X, \Delta)>-1$ for every exceptional divisor $E$ over $X$, 
then $(X, \Delta)$ is called a {\em{plt}} pair. 
If $a(E, X, \Delta)>-1$ for every divisor $E$ over $X$, then $(X, \Delta)$ is 
called a {\em{klt}} pair. 
In this paper, if $\Delta=0$ and $a(E, X, 0)\geq 0$ for 
every divisor $E$ over $X$, then we say that $X$ has only {\em{canonical 
singularities}}. 
\end{say}

\begin{rem}\label{rem2.3} 
Although $\mathbb R$-divisors play crucial roles in the 
recent developments of the minimal model program, 
we do not use $\mathbb R$-divisors in this paper. 
\end{rem}

We need the following lemma in the proof of 
the local freeness in the main theorem:~Theorem \ref{thm1.6}. 

\begin{lem}\label{lem2.4}
Let $X$ be a normal variety with only canonical singularities. 
Then $\mathcal O_X(mK_X)$ is Cohen--Macaulay for every integer $m$. 
\end{lem}

\begin{proof}
We note that $X$ has only rational singularities when $X$ is canonical. 
Let $r$ be the smallest positive integer such that 
$rK_X$ is Cartier. 
Since the problem is local, we may assume that 
$rK_X\sim 0$ by shrinking $X$. 
If $r=1$, then $\mathcal O_X(mK_X)\simeq \mathcal O_X$ for 
every integer $m$. 
In this case, $\mathcal O_X(mK_X)$ is Cohen--Macaulay for every 
integer $m$ since $X$ has only rational singularities. 
From now on, we assume that 
$r\geq 2$. 
Let $\pi:\widetilde X\to X$ be the index one cover. 
Then we have 
$$
\pi_*\mathcal O_{\widetilde X}(K_{\widetilde X})\simeq 
\bigoplus _{i=1}^r \mathcal O_X(iK_X). 
$$ 
Since $\widetilde X$ has only canonical singularities and 
$K_{\widetilde X}$ is Cartier, 
$\mathcal O_{\widetilde X}(K_{\widetilde X})$ is Cohen--Macaulay. 
Since $\pi$ is finite, 
$\mathcal O_X(iK_X)$ is Cohen--Macaulay for $1\leq i\leq r$. 
By $rK_X\sim 0$, we obtain that 
$\mathcal O_X(mK_X)$ is Cohen--Macaulay for 
every integer $m$. 
\end{proof}

\section{Relative good minimal models}\label{sec3}
In this section, we discuss the relationship between 
relative good minimal models and 
good minimal models of fibers for the reader's convenience. 
The results in this section are more or less known to the experts 
although they were not stated explicitly in the literature. 

Let us recall the definition of sufficiently general fibers. 

\begin{defn}[Sufficiently general fibers]\label{def3.1}
Let $f:X\to Y$ be a morphism between algebraic varieties. 
Then a {\em{sufficiently general fiber}} $F$ of $f:X\to Y$ means that 
$F=f^{-1}(y)$ where $y$ is any point contained in a countable 
intersection of Zariski dense open subsets of $Y$. 
\end{defn}

A sufficiently general fiber is sometimes called a {\em{very general fiber}} 
in the literature. 

\begin{defn}[Good minimal models]\label{def3.2}
Let $f:X\to Y$ be a projective morphism between normal 
quasi-projective 
varieties. 
Let $\Delta$ be an effective $\mathbb Q$-divisor on $X$ such that 
$(X, \Delta)$ is klt. 
A pair $(X', \Delta')$ sitting in a diagram 
$$
\xymatrix{
X\ar[dr]_{f}\ar@{-->}[rr]^\phi&& X'\ar[dl]^{f'}\\
&Y&
}
$$
is called a {\em{minimal model of $(X, \Delta)$ over $Y$}} if 
\begin{itemize}
\item[(i)] $X'$ is $\mathbb Q$-factorial, 
\item[(ii)] $f'$ is projective, 
\item[(iii)] $\phi$ is birational and $\phi^{-1}$ has no 
exceptional divisors, 
\item[(iv)] $\phi_*\Delta=\Delta'$, 
\item[(v)] $K_{X'}+\Delta'$ is $f'$-nef, and 
\item[(vi)] $a(E, X, \Delta)<a(E, X', \Delta')$ for every $\phi$-exceptional 
divisor $E\subset X$. 
\end{itemize}
Furthermore, if $K_{X'}+\Delta'$ is $f'$-semi-ample, 
then $(X', \Delta')$ is called a {\em{good minimal model 
of $(X, \Delta)$ over $Y$}}. 
When $Y$ is a point, we usually omit \lq\lq over $Y$\rq\rq \ in the above definitions. 
We sometimes simply say that $(X', \Delta')$ is a {\em{relative 
$($good$)$ minimal model}} of $(X, \Delta)$. 
\end{defn}

Although Theorem \ref{thm3.3} holds for klt pairs, we state it for 
varieties with only canonical singularities for simplicity. 
Theorem \ref{thm3.3} is useful and sufficient for our application in 
this paper. 

\begin{thm}\label{thm3.3} 
Let $f:X\to Y$ be a projective surjective morphism 
from a normal quasi-projective variety $X$ with 
only canonical singularities to 
a normal quasi-projective variety $Y$ with connected fibers. 
Then the following conditions are equivalent. 
\begin{itemize}
\item[(i)] $X$ has a good minimal model over $Y$. 
\item[(ii)] $X_{\overline \eta}$ has a good minimal model, 
where $X_{\overline \eta}$ is the geometric generic fiber of $f:X\to Y$. 
\item[(iii)] $F$ has a good minimal model, where $F$ is 
a sufficiently general fiber of $f:X\to Y$. 
\item[(iv)] $G$ has a good minimal model, where $G$ is a 
general fiber of $f:X\to Y$. 
\end{itemize}
\end{thm}

In order to understand Theorem \ref{thm3.3}, we give some supplementary 
results. 

\begin{thm}\label{thm3.4}
Let $(X, \Delta)$ be a projective klt pair such that 
$\Delta$ is a $\mathbb Q$-divisor. 
Then $(X, \Delta)$ has a good minimal model 
if and only if 
$K_X+\Delta$ is pseudo-effective, equivalently, $\kappa _\sigma (X, K_X+\Delta)\geq 0$, 
and 
$$\kappa (X, K_X+\Delta)=\kappa_{\sigma}(X, K_X+\Delta),$$ 
where $\kappa_\sigma$ denotes Nakayama's numerical Kodaira dimension 
and $\kappa$ denotes Iitaka's $D$-dimension.  
\end{thm}
\begin{proof}
For the proof, see \cite[Theorem 4.3]{gongyo-lehmann} or \cite[Remark 2.6]{dhp}. 
\end{proof}

\begin{cor}\label{cor3.5} 
Let $V$ be a smooth projective variety and 
let $V'$ be a normal projective variety with only canonical singularities 
such that $V$ is birationally equivalent to $V'$. 
Then $V$ has a good minimal model if and only if $V'$ 
has a good minimal model. 
\end{cor}
\begin{proof}
Note that $\kappa(V, K_V)=\kappa (V', K_{V'})$ and 
$\kappa _{\sigma}(V, K_V)=\kappa_{\sigma}(V', K_{V'})$ hold since 
$V'$ has only canonical singularities. 
Therefore, we see that $\kappa (V, K_V)=\kappa_{\sigma}(V, K_V)$ if and 
only if $\kappa(V', K_{V'})=\kappa_{\sigma}(V', K_{V'})$. 
By Theorem \ref{thm3.4}, we have the desired statement. 
\end{proof}

\begin{lem}\label{lem3.6}
Let $f:X\to Y$ be a projective surjective morphism between normal 
varieties with connected 
fibers and let $\Delta$ be an effective $\mathbb Q$-divisor on $X$ such that 
$(X, \Delta)$ is klt. 
Let $X_{\overline \eta}$ be the geometric generic fiber of $f:X\to Y$. 
We put $\Delta_{\overline \eta}=\Delta|_{X_{\overline \eta}}$. 
Then we have 
$$
\kappa (X_{\overline \eta}, K_{X_{\overline \eta}}+\Delta_{\overline 
\eta})=\kappa (F, K_F+\Delta|_F)
$$ and 
$$
\kappa_\sigma (X_{\overline \eta}, K_{X_{\overline \eta}}+\Delta_{\overline 
\eta})=\kappa _\sigma(F, K_F+\Delta|_F)
$$ 
where $F$ is a sufficiently general fiber of $f:X\to Y$. 
\end{lem}
\begin{proof}
This is obvious by the definitions of Iitaka's $D$-dimension 
$\kappa$ and Nakayama's numerical Kodaira dimension $\kappa_{\sigma}$. 
For the details, see \cite{nakayama2} and \cite{lehmann}. 
\end{proof}

By combining Theorem \ref{thm3.4} with Lemma \ref{lem3.6}, 
we have: 

\begin{cor}\label{cor3.7}
Let $f:X\to Y$ be a projective surjective morphism 
between normal varieties and let $\Delta$ be 
an effective $\mathbb Q$-divisor on $X$ such that 
$(X, \Delta)$ is klt. 
Then $(X_{\overline \eta}, \Delta_{\overline \eta})$ has 
a good minimal model if and only if 
$(F, \Delta|_F)$ has a good minimal model, where 
$F$ is a sufficiently general fiber of $f:X\to Y$. 
\end{cor}

\begin{proof}
This statement is obvious by Theorem \ref{thm3.4} and 
Lemma \ref{lem3.6}. 
\end{proof}

Let us give a proof of Theorem \ref{thm3.3} 
for the reader's convenience. 

\begin{proof}[Proof of Theorem \ref{thm3.3}] 
We divide the proof into several steps. 

\begin{step}[(ii)$\Longleftrightarrow$(iii)]\label{Tstep1}
This step is a special case of Corollary \ref{cor3.7}. 
\end{step}
\begin{step}[(iv)$\Longrightarrow$(iii)]\label{Tstep2}
This is obvious since a sufficiently general fiber of $f:X\to Y$ is 
a general fiber of $f:X\to Y$. 
\end{step}
\begin{step}[(i)$\Longrightarrow$(iv)]\label{Tstep3}
We consider the following commutative diagram 
$$
\xymatrix{
X\ar[dr]_{f}\ar@{-->}[rr]^\phi&& X'\ar[dl]^{f'}\\
&Y&
}
$$
where $f': X'\to Y$ is a good minimal model of $X$ over $Y$. 
We take a general point $y\in Y$. 
Let us 
consider the diagram 
$$
\xymatrix{
G\ar[dr]_{f}\ar@{-->}[rr]^\psi&& G'\ar[dl]^{f'}\\
&y&
}
$$ 
where $G=f^{-1}(y)$, $G'=f'^{-1}(y)$, and $\psi=\phi|_G$. 
Since $y\in Y$ is a general point, the 
above diagram satisfies the conditions (ii), (iii), 
(iv), (v), and (vi) in Definition \ref{def3.2}. 
Moreover, $K_{G'}$ is semi-ample because $K_{X'}$ is $f'$-semi-ample. 
If $G'$ is not $\mathbb Q$-factorial, 
then we replace $G'$ with its small projective 
$\mathbb Q$-factorialization. 
Then $G'$ also satisfies the condition (i) in Definition \ref{def3.2} and 
is a good minimal model of $G$. 
\end{step}
\begin{step}[(iii)$\Longrightarrow$(i)]\label{Tstep4}
This is a special case of \cite[Theorem 2.12]{hacon-xu} 
(see also the proof of \cite[Theorem 5.1]{birkar}). 
\end{step}
We have completed the proof of Theorem \ref{thm3.3}. 
\end{proof}

We close this section with 
a useful result, which follows from \cite[Theorem 4.4]{lai} 
(see also \cite[Theorem 1.5]{birkar} and \cite[Theorem 2.12]{hacon-xu}). 

\begin{thm}\label{thm3.8} 
Let $X$ be a smooth projective variety with non-negative 
Kodaira dimension. Then $X$ has a good minimal model if and only if 
the geometric generic fiber of the Iitaka fibration of $X$ has a good minimal 
model. 
\end{thm}

\begin{proof}
See \cite[Theorem 4.4]{lai}, \cite[Theorem 5.1]{birkar}, and 
\cite[Theorem 2.12]{hacon-xu}. 
\end{proof}

By Theorem \ref{thm3.8}, 
we know that any smooth projective variety $X$ with 
$\dim X-\kappa (X)\leq 3$ has a good minimal model. 

\begin{rem}\label{rem3.9} 
In the proof of Theorem \ref{thm3.8} and Step \ref{Tstep4} in the 
proof of Theorem \ref{thm3.3}, 
we need the finite generation of canonical rings for 
(relative) klt pairs, which is established in \cite{bchm}. 
We note that the final step of the proof of the finite generation of 
canonical rings for klt pairs needs the canonical bundle formula 
due to Fujino--Mori (see \cite{fujino-mori}). 
We also note that the canonical bundle formula treated in \cite{fujino-mori} 
depends on Theorem \ref{thm1.3}. 
Therefore, our proof of Theorem \ref{thm1.6} in this paper 
implicitly uses Theorem \ref{thm1.3}. 
\end{rem}

\section{Local freeness of $f'_*\omega^{\otimes m}_{X'/Y'}$}\label{sec4}

In this section, we prove the local freeness of $f'_*\omega^{\otimes m}_{X'/Y'}$ 
in Theorem \ref{thm1.6} by using 
minimal model theory and the weak semistable reduction theorem 
due to 
Abramovich--Karu (see \cite{abramovich-karu}). 

Let us start with the proof of the local freeness of 
$f_*\omega^{\otimes m}_{X/Y}$ 
in Theorem \ref{thm1.4}. It is a direct consequence of 
Siu's invariance of plurigenera (see \cite[Corollary 0.2]{siu} and \cite
[Theorem 1]{paun}). 

\begin{proof}[Proof of the local freeness of $f_*\omega^{\otimes m}_{X/Y}$ 
in Theorem \ref{thm1.4}]
By \cite[Corollary 0.2]{siu}, we know that 
$$
\dim H^0(X_y, \mathcal O_{X_y}(mK_{X_y}))
$$ 
is independent of $y\in Y$ for every positive integer $m$ 
(see also \cite[Theorem 1]{paun}). 
By the base change theorem (see \cite[Chapter III, Corollary 12.9]{hartshorne}), 
this implies that 
$f_*\omega^{\otimes m}_{X/Y}$ is locally free for 
every $m\geq 1$. 
\end{proof}

Let us recall the following well-known lemma, which is 
a special case of \cite[Corollary 3]{nakayama1}. 

\begin{lem}[{cf.~\cite[Corollary 3]{nakayama1}}]\label{lem4.1} 
Let $g:V\to C$ be a projective surjective morphism 
from a normal quasi-projective 
variety $V$ to a smooth 
quasi-projective curve $C$. 
Assume that $V$ has only canonical singularities and 
that $K_V$ is $g$-semi-ample. 
Then $R^ig_*\mathcal O_V(mK_V)$ is locally free for every $i$ and 
every positive integer $m$. 
\end{lem}
\begin{proof}
Let $h:V'\to V$ be a resolution of singularities such that 
$\Exc(h)$ is a simple normal crossing divisor on $V'$. 
We write 
$$
K_{V'}=h^*K_V+E,
$$ 
where $E$ is an effective $h$-exceptional 
$\mathbb Q$-divisor. 
Then we have 
$$
\lceil mh^*K_V+ E\rceil -(K_{V'}+\{-(mh^*K_V+E)\})=(m-1)h^*K_V. 
$$ 
We note that the right hand side is semi-ample over $C$. 
Therefore, $$R^i(g\circ h)_*\mathcal O_{V'}(\lceil mh^*K_V+ 
E\rceil)$$ is locally free for every $i$ and every positive integer 
$m$ (see, for example, \cite[Theorem 6.3 (i)]{fujino-fund}). 
On the other hand, we have 
$$
R^ih_*\mathcal O_{V'}(\lceil mh^*K_V+ E\rceil)=0 
$$ 
for every $i>0$ by the relative Kawamata--Viehweg vanishing theorem, 
and 
$$
h_*\mathcal O_{V'}(\lceil mh^*K_V+
E\rceil)\simeq \mathcal O_V(mK_V). 
$$
Therefore, we obtain that 
$$
R^ig_*\mathcal O_V(mK_V)$$ is locally free for every $i$ and 
every positive integer $m$. 
\end{proof}

\begin{proof}[Proof of the local freeness of $f'_*\omega^{\otimes m}_{X'/Y'}$ 
in Theorem \ref{thm1.6}] Let us divide the proof into several steps. 
\setcounter{step}{0}
\begin{step}[Weak semistable reduction]\label{step1}
By \cite[Definition 0.1 and Theorem 0.3]{abramovich-karu}, 
there exist a generically finite morphism $\tau:Y'\to Y$ from a smooth 
projective variety $Y'$ 
and $f^\dag: X^\dag\to Y'$  
with the following properties. 
\begin{itemize}
\item[(i)] $X^\dag$ is a normal projective 
Gorenstein (see \cite[Lemma 6.1]{abramovich-karu}) 
variety which is birationally equivalent to 
$X\times _Y Y'$. 
\item[(ii)] $(U_{X^\dag}\subset X^\dag)$ and 
$(U_{Y'}\subset Y')$ are toroidal embeddings without self-intersection, with 
$U_{X^\dag} =(f^\dag)^{-1}(U_{Y'})$. 
\item[(iii)] $f^\dag: (U_{X^\dag}\subset X^{\dag})
\to (U_{Y'}\subset Y')$ is toroidal and equidimensional. 
\item[(iv)] all the fibers of the morphism $f^\dag$ are reduced. 
\end{itemize}
In \cite{abramovich-karu}, $f^\dag: X^{\dag}\to Y'$ is said to be weakly semistable and is 
called a 
weak semistable reduction of $f:X\to Y$. 
For the details of toroidal embeddings and 
morphisms, see \cite[Section 1]{abramovich-karu}. 
We may further assume that $X^\dag$ is 
$\mathbb Q$-factorial (see \cite[Remark 4.3]{abramovich-karu}).  
Note that $X^{\dag}$ has only rational singularities 
since $X^{\dag}$ is toroidal. 
Therefore, $X^{\dag}$ has only 
canonical Gorenstein singularities and is Cohen--Macaulay. 
Thus, we have $$
f^\dag_*\mathcal O_{X^{\dag}}(mK_{X^{\dag}/Y'})\simeq f'_*\omega^{\otimes m}_{X'/Y'}
$$ 
for every positive integer $m$. 
Therefore, it is sufficient to prove 
that $f^\dag_*\mathcal O_{X^{\dag}}(mK_{X^{\dag}/Y'})$ is locally 
free for every positive integer $m$. 
We also note that $f^\dag$ is flat since $Y'$ is smooth, $X^\dag$ is 
Cohen--Macaulay, and $f^\dag$ is equidimensional 
(see \cite[Chapter III, Exercise 10.9]{hartshorne} 
and \cite[Chapter V, Proposition (3.5)]{altman-kleiman}). 
\end{step}

\begin{rem}
We may assume that $f^\dag$ is smooth 
over $U_{Y'}$ although we do not need this property in this paper. 
For the details, see the construction of weak semistable reductions 
in \cite{abramovich-karu}. 
\end{rem}

\begin{step}[Relative good minimal models]\label{step2}
By the assumption of Theorem \ref{thm1.6} and Corollary \ref{cor3.5}, 
the geometric generic fiber of $f^\dag: X^\dag\to Y'$ has a good minimal 
model. Therefore, $f^\dag: X^\dag\to Y'$ has a relative good minimal 
model $\widetilde f: \widetilde X\to Y'$ 
by Theorem \ref{thm3.3}. 
Note that 
$$
f^\dag_*\mathcal O_{X^{\dag}}(mK_{X^{\dag}/Y'})\simeq 
\widetilde f_*\mathcal O_{\widetilde X}(mK_{{\widetilde X}/Y'})
$$ 
for every positive integer $m$. 
Therefore, it is sufficient to prove 
that 
${\widetilde f}_*\mathcal O_{\widetilde X}(mK_{\widetilde {X}/Y'})$ is locally 
free for every positive integer $m$. 
\end{step}
\begin{step}[Local freeness via the 
flat base change theorem]\label{step3}
We take an arbitrary point $P\in Y'$. 
We take general very ample Cartier divisors 
$H_1, H_2, \cdots, H_{n-1}$, where $n=\dim Y$, such that 
$C=H_1\cap H_2\cap \cdots \cap H_{n-1}$ is a smooth 
projective curve passing 
through $P$. 
By \cite[Lemma 6.2]{abramovich-karu}, 
we see that $X^\dag_C=X^\dag\times _{Y'}C\to C$ is weakly 
semistable. 
In particular, $X^\dag_C$ has only rational Gorenstein singularities 
(see \cite[Lemma 6.1]{abramovich-karu}). 
By adjunction, 
we see that 
${\widetilde X}_C=\widetilde X \times _{Y'}C$ is normal and 
has only canonical singularities. 
More precisely, 
$(f^\dag)^*H_1= X^\dag\times _{Y'} H_1= X^\dag_{H_1}$ has 
only rational Gorenstein singularities since 
$X^\dag_{H_1}\to H_1$ is weakly 
semistable by \cite[Lemma 6.1 and 
Lemma 6.2]{abramovich-karu}. 
In particular, $(f^\dag)^*H_1$ has only canonical singularities. 
Therefore, 
$(X^\dag, (f^\dag)^*H_1)$ is plt by the inversion of 
adjunction (see \cite[Theorem 5.50]{kollar-mori}). 
So we have that $(\widetilde X, {\widetilde f}^*H_1)$ is 
plt by the negativity lemma (see, for example, \cite[Proposition 3.51]{kollar-mori}). 
Thus, $\widetilde X_{H_1}=\widetilde X\times _{Y'} H_1={\widetilde f}
^*H_1$ is 
normal (see \cite[Proposition 5.51]{kollar-mori}). 
By adjunction and the negativity lemma again, 
we obtain that ${\widetilde X}_{H_1}$ has only 
canonical singularities. By repeating this process $(n-1)$-times, 
we obtain that 
${\widetilde X}_C$ has only canonical singularities.  
Note that $\widetilde X_C\to C$ is equidimensional. 
Therefore, 
we see that 
$\widetilde f: \widetilde X\to Y'$ is equidimensional by the 
choice of $C$.  
Since $\widetilde X$ is Cohen--Macaulay and $Y'$ is smooth, 
$\widetilde f$ is flat (see \cite[Chapter III, Exercise 10.9]{hartshorne} 
and \cite[Chapter V, Proposition (3.5)]{altman-kleiman}). 
Moreover, $\mathcal O_{\widetilde X}(mK_{\widetilde X})$ is 
flat over $Y'$ for every integer $m$ since $\mathcal O_{\widetilde X}(m
K_{\widetilde X})$ is Cohen--Macaulay by Lemma \ref{lem2.4} and 
$\widetilde f$ is equidimensional (see \cite[Chapter V, Proposition (3.5)]
{altman-kleiman}). 
By applying Lemma \ref{lem4.1} and 
the base change theorem (see \cite[Chapter III, Theorem 12.11]{hartshorne}) 
to $\widetilde X_C\to C$, 
we obtain that 
$$
\dim H^0({\widetilde X}_y, 
\mathcal O_{\widetilde X}(mK_{\widetilde X/Y'})|_{\widetilde X_y})
$$ 
is independent of $y\in Y'$ for every positive integer $m$. 
By the base change theorem (see \cite[Chapter III, Corollary 12.9]{hartshorne}), 
we obtain that 
$f'_*\omega^{\otimes m}_{X'/Y'} 
\simeq 
{\widetilde f}_*\mathcal O_{\widetilde X}(mK_{\widetilde {X}/Y'})$ is 
locally free for every positive integer $m$. 
\end{step}
We have 
completed the proof of the local freeness of $f'_*\omega^{\otimes m}_{X'/Y'}$. 
\end{proof}

\begin{rem}
In general, $\widetilde X_y$ may be non-normal. 
However, we see that the canonical 
divisor $K_{\widetilde X_y}$ is well-defined, 
$\widetilde X_y$ has only semi log canonical singularities, 
and  $\mathcal O_{\widetilde X}(mK_{\widetilde X/Y'})|_{\widetilde X_y}
\simeq \mathcal O_{\widetilde X_y}(mK_{\widetilde X_y})$ 
for every positive integer $m$, by 
adjunction. 
For the details of semi log canonical singularities and pairs, 
see \cite{fujino-slc}. 
\end{rem}

In Step \ref{step3} in the proof of the local freeness of $f'_*\omega^{\otimes m}
_{X'/Y'}$ in Theorem \ref{thm1.6}, we have proved: 

\begin{thm}\label{thm4.2}
Let $\pi:V\to W$ be a projective surjective morphism between 
quasi-projective varieties with connected fibers. 
Assume the following conditions: 
\begin{itemize}
\item[(i)] $W$ is smooth, 
\item[(ii)] $(U_V\subset V)$ and $(U_W\subset W)$ are 
toroidal embeddings without self-intersection, with $U_V=\pi^{-1}(U_W)$,  
\item[(iii)] $f:(U_V\subset V)\to (U_W\subset W)$ is toroidal and 
equidimensional, and 
\item[(iv)] all the fibers of the morphism $\pi$ are reduced. 
\end{itemize} 
In this case, $\pi:V\to W$ is said to be weakly semistable. 
We know that $V$ has only rational Gorenstein singularities. 
Let $V'$ be a minimal model of $V$ over $W$ 
sitting in the following diagram. 
$$
\xymatrix{
V\ar[dr]_{\pi}\ar@{-->}[rr]^\phi&& V'\ar[dl]^{\pi'}\\
&W&
}
$$
Let $P\in W$ be an arbitrary point and let $C$ be a smooth 
curve such that $P\in C$ and that $C=H_1\cap H_2\cap 
\cdots \cap H_{\dim W-1}$, 
where $H_i$ is a general smooth Cartier divisor on $W$ for every $i$. 
Then $V_C=V\times _W C\to C$ is 
weakly semistable and 
$V'_C=V'\times _W C$ is normal and has only canonical singularities. 
This implies that $\pi': V'\to W$ is equidimensional. 
In particular, $\pi'$ is flat. 
\end{thm}

Theorem \ref{thm4.2} seems to be useful for various geometric applications. 
So we wrote it separately for the reader's convenience. 
Note that Theorem \ref{thm4.2} (see also Step \ref{step3} in the 
proof of the local freeness of $f'_*\omega^{\otimes m}_{X'/Y'}$ in 
Theorem \ref{thm1.6}) is a key point of this paper. 

\section{Nefness of $f'_*\omega^{\otimes m}_{X'/Y'}$}\label{sec5}

In this section, we prove that $f'_*\omega^{\otimes m}_{X'/ Y'}$ 
in Theorem \ref{thm1.6} is nef (numerically semipositive) 
by using \cite{popa-schnell}. 
We do not use the theory of variations of Hodge structure.   
Theorem \ref{thm5.1}, which is 
a key ingredient of this section, follows from \cite[Theorem 1.4]{popa-schnell}. 

\begin{thm}\label{thm5.1} 
Let $f:X\to Y$ be a surjective morphism between smooth projective 
varieties with connected fibers. 
Let $\mathcal L$ be an ample and globally generated 
line bundle on $Y$ and let $k$ be a positive integer. 
Then $$
f_*\omega^{\otimes k} _X\otimes \mathcal L^{\otimes l}\simeq 
f_*\omega^{\otimes k}_{X/Y}\otimes \omega^{\otimes k}_Y\otimes 
\mathcal L^{\otimes l}
$$ 
is generated by global sections for $l\geq k(\dim Y+1)$. 
\end{thm}
\begin{proof}
See \cite[Section 2]{popa-schnell}. 
\end{proof}

\begin{rem}\label{rem5.2}
Theorem \ref{thm5.1} holds under the weaker assumption that 
$X$ is a normal projective variety with only rational Gorenstein singularities. 
Note that $X$ has only rational Gorenstein singularities 
if and only if $X$ has only canonical Gorenstein singularities. 
\end{rem}

\begin{lem}\label{lem5.3}
Let $\mathcal E$ be 
a non-zero locally free sheaf of finite rank on a smooth 
projective variety $V$. 
Assume that 
there exists a line bundle $\mathcal M$ 
such that $\mathcal E^{\otimes s}\otimes \mathcal M$ 
is 
generated by global sections for every positive integer $s$. 
Then $\mathcal E$ is nef. 
\end{lem}

\begin{proof}
We put $\pi:W=\mathbb P_V(\mathcal E)\to V$ and 
$\mathcal O_W(1)=\mathcal O_{\mathbb P_V(\mathcal E)}(1)$. 
Since $\mathcal E^{\otimes s}\otimes \mathcal M$ is generated by 
global sections, 
$\mathrm{Sym}^s\mathcal E\otimes \mathcal M$ is 
also generated by global sections for every positive integer $s$. 
This implies that $\mathcal O_W(s)\otimes \pi^*\mathcal M$ is 
generated by global sections for every positive integer $s$. 
Thus, we obtain that $\mathcal O_W(1)$ is nef, equivalently, 
$\mathcal E$ is nef. 
\end{proof}

Let us prove the nefness of $f_*\omega^{\otimes m}_{X/Y}$ in 
Theorem \ref{thm1.4}. 

\begin{proof}[Proof of the nefness of $f_*\omega^{\otimes m}_{X/Y}$ 
in Theorem \ref{thm1.4}] 
We take the $s$-fold fiber product 
$$f^s:X^s=X\times _YX\times _Y\cdots \times _Y X\to Y. 
$$ 
Since $f$ is smooth, 
$X^s$ is a smooth projective variety and $f^s$ is smooth. 
We will check 
$$
f^s_*\omega^{\otimes m}_{X^s/Y}\simeq \bigotimes ^{s} 
f_*\omega^{\otimes m}_{X/Y} 
$$ 
for every positive integer $m$ by induction on $s$. 
We consider the following commutative diagram: 
$$
\xymatrix{
X^s\ar[r]^p\ar[d]_q&\ X^{s-1}\ar[d]^{f^{s-1}}\\
X\ar[r]_f & Y. 
}
$$ 
By  base change, 
we have $\omega_{X^s/X}\simeq p^*\omega_{X^{s-1}/Y}$. 
Thus we have 
\begin{align*}
\omega_{X^s/Y}&\simeq \omega_{X^s/X}\otimes q^*\omega_{X/Y} 
\\&\simeq p^*\omega_{X^{s-1}/Y}\otimes q^*\omega_{X/Y}. 
\end{align*}
Therefore, by the flat base change theorem 
(see \cite[Chapter III, Proposition 9.3]{hartshorne}) 
and the projection formula, we obtain 
\begin{align*}
f^s_*\omega^{\otimes m}_{X^s/Y}&\simeq 
f^{s-1}_*p_*(p^*\omega^{\otimes m}_{X^{s-1}/Y}
\otimes q^*\omega^{\otimes m}_{X/Y})\\
&\simeq f^{s-1}_*(\omega^{\otimes m}_{X^{s-1}/Y}
\otimes p_*q^*\omega^{\otimes m}_{X/Y})\\
&\simeq  f^{s-1}_*(\omega^{\otimes m}_{X^{s-1}/Y}
\otimes (f^{s-1})^*f_*\omega^{\otimes m}_{X/Y})\\
&\simeq f_*\omega^{\otimes m}_{X/Y}\otimes 
f^{s-1}_*\omega^{\otimes m}_{X^{s-1}/Y}\\
&\simeq \bigotimes ^s f_*\omega^{\otimes m}_{X/Y}
\end{align*} 
for every positive integer $m$ and every positive integer $s$ by 
induction on $s$. Note that 
$f_*\omega^{\otimes m}_{X/Y}$ is locally free for every positive integer $m$ 
(see Section \ref{sec4}). 
We put $\mathcal M=\omega^{\otimes m}_Y\otimes \mathcal L^{\otimes 
m(\dim Y+1)}$, where 
$\mathcal L$ is an ample and 
globally generated line bundle on $Y$. 
By Theorem \ref{thm5.1}, 
we obtain that 
$$
f^s_*\omega^{\otimes m}_{X^s/Y} \otimes \mathcal M
$$ 
is generated by global sections for every positive integer $s$. 
This means that 
$$
\bigotimes ^{s} 
f_*\omega^{\otimes m}_{X/Y} 
\otimes \mathcal M
$$ is generated by global sections for every positive integer $s$. 
By Lemma \ref{lem5.3}, 
we obtain that $f_*\omega^{\otimes m}_{X/Y}$ is nef. 
\end{proof}

In Section \ref{sec4}, we have already proved that 
$f'_*\omega^{\otimes m}_{X'/Y'}$ 
is locally free in Theorem \ref{thm1.6}. 
From now on, we prove the nefness of 
$f'_*\omega^{\otimes m}_{X'/Y'}$ in Theorem \ref{thm1.6}. 

\begin{proof}[Proof of the nefness of 
$f'_*\omega^{\otimes m}_{X'/Y'}$ in Theorem \ref{thm1.6}]
By the proof of the local freeness of $f'_*\omega^{\otimes m}_{X'/Y'}$ in 
Section \ref{sec4}, we may assume that $f': X'\to Y'$ is 
weakly semistable. 
For simplicity, we denote $f':X'\to Y'$ by $f:X\to Y$ in this proof.  
We take the $s$-fold fiber product 
$$
f^s:X^s=X\times _YX\times _Y \cdots \times _YX \to Y. 
$$ 
Then we see that $X^s$ is normal and Gorenstein (cf.~\cite[Lemma 3.5]{viehweg}). 
Moreover, $X^s$ has only rational singularities because 
$X^s$ is local analytically isomorphic to 
a toric variety. 
Therefore, $X^s$ has only canonical singularities. 
By the same argument as in the proof of 
the nefness of $f_*\omega^{\otimes m}_{X/Y}$ in 
Theorem \ref{thm1.4}, we obtain 
$$
f^s_*\omega^{\otimes m}_{X^s/Y}\simeq 
\bigotimes ^s f_*\omega^{\otimes m}_{X/Y}
$$ for every positive integer $s$ and every positive integer $m$. 
By Theorem \ref{thm5.1} (see also Remark \ref{rem5.2}) 
and Lemma \ref{lem5.3}, 
we obtain that the locally free sheaf $f_*\omega^{\otimes m}_{X/Y}$ is 
nef for every positive integer $m$. 
This is the same as the proof of the nefness of $f_*\omega^{\otimes m}
_{X/Y}$ in Theorem \ref{thm1.4}. 
Anyway, we obtain the desired nefness. 
\end{proof}

\end{document}